\documentclass[12pt,reqno]{elsarticle}

\usepackage{amsfonts,color,amsmath,amssymb,fancyhdr}
\usepackage{amsthm}
\usepackage{tikz}
\usepackage{graphicx}

\usepackage[colorlinks,linkcolor=blue, anchorcolor=red,citecolor=green]{hyperref}
\usepackage{subfigure}

\def\Box{\vcenter{\vbox{\hrule\hbox{\vrule
     \vbox to 8.8pt{\hbox to 10pt{}\vfill}\vrule}\hrule}}}

\newcommand{\F}{{\mathbb F}}

\newcommand{\PGL}{\textup{PGL}}

\newcommand{\GL}{\textup{GL}}

\newcommand{\diag}{\textup{diag}}

\newcommand{\tr}{\textup{Tr}}

\newtheorem{thm}{Theorem}

\newtheorem{lemma}[thm]{Lemma}
\newtheorem{corollary}[thm]{Corollary}

\numberwithin{equation}{section}
\numberwithin{thm}{section}

\theoremstyle{definition}

\begin{document}
\newcommand{\stopthm}{\begin{flushright}
		\(\box \;\;\;\;\;\;\;\;\;\; \)
\end{flushright}}
\newcommand{\symfont}{\fam \mathfam}

\title{On codes in the projective linear group $\PGL(2,q)$}

\date{}
\author[add1]{Tao Feng}\ead{tfeng@zju.edu.cn}
\author[add1]{Weicong Li}\ead{conglw@zju.edu.cn}
\author[add1]{Jingkun Zhou\corref{cor1}}\ead{jingkunz@zju.edu.cn}\cortext[cor1]{Corresponding author}
\address[add1]{School of Mathematical Sciences, Zhejiang University, 38 Zheda Road, Hangzhou 310027, Zhejiang P.R China}

\begin{abstract}
  In this paper, we resolve a conjecture of Green and Liebeck [Disc. Math., 343 (8):117119, 2019] on codes in $\PGL(2,q)$. To be specific, we show that: if $D$ is a dihedral subgroup of order $2(q+1)$ in $G=\PGL(2,q)$, and  $A=\{g\in G: g^{q+1}= 1,\, g^2\ne 1 \}$, then $\lambda G=A\cdot D$, where $\lambda=q$ or $q-1$ according as $q$ is even or odd.
	\newline
	
	\noindent\text{Keywords:} codes, Cayley graphs, projective linear groups, linearized polynomials.
	
	\noindent\text{Mathematics Subject Classification (2010)}: 05C25 94B60 11T06
\end{abstract}	

\maketitle

\section{Introduction}

Let $G$ be a finite group, $A$ be a nonempty proper subset of $G$ and $\lambda$ be a natural number. Following \cite{Tedera2004}, we say that $A$ \textit{divides} $\lambda G$ if there is a subset $B$ of $G$ such that the multiset $\{ab:\,a\in A,\, b \in B\}$ covers each element of $G$ exactly $\lambda$ times; the subset $B$ is called a \textit{code} with respect to $A$ and we write $A\cdot B=\lambda G$. If the subset $A$ of $G$ does not contain the identity and satisfies that $A=\{g^{-1}:\,g\in A\}$, we define the \textit{Cayley graph} $\textup{Cay}(G,A)$ as the graph with vertex set $G$ and edge set $\{(g,h):\,g^{-1}h\in A\}$. In such a case, $A\cdot B=\lambda G$ implies that the code $B$ is a collection of vertices of  $\textup{Cay}(G,A)$ such that each vertex has exactly $\lambda$ neighbors in $B$. If $\lambda=1$, then $B$ is called a \textit{perfect code} of $\textup{Cay}(G,A)$. This notion of perfect codes generalizes the classical perfect codes in coding theory.
The recent paper \cite{Huang2018} gives a very nice exposition of their relationship and the history of research development in this direction. We refer the reader to
\cite{Biggs1973,Lint1975,Jan1986} for the classical theory of perfect codes over graphs and \cite{Huang2018,Green2019,Ma2019,Chen2020} for some recent progress.

When $A$ is the union of conjugacy classes, it is of  particular interest to study codes in $\textup{Cay}(G,A)$. In \cite{Et87, Tedera2004}, the authors use representation theory to study such codes.  However, there are very few known examples in the literature, and this motivated the authors of \cite{Green2019}  to study the case where $A$ is a union of conjugacy classes and $B$ is a subgroup. They construct several families of such codes in symmetric groups and special linear groups $\textup{SL}(2,q)$ and pose two conjectured families. This paper resolves their Conjecture 3.2, which we now state as a theorem.

\begin{thm}\label{conj_main}
Let $G=\PGL(2,q)$ and $D$ be a dihedral subgroup of order $2(q+1)$. Set
\[
A=\{g\in G: g^{q+1}= 1,\, g^2\ne 1 \}.
\]
Then $\lambda  G= A\cdot D$, where $\lambda= q$ or $q-1$ according as $q$ is even or odd.
\end{thm}

This paper is organized as follows. In Section 2, we describe the model of $\GL(2,q)$ that we use and present some preliminary results about the the preimages of $A$ and $D$ in $\GL(2,q)$, where $A,\,D$ are the subsets in Theorem \ref{conj_main}. We shall prove the theorem by working inside $\GL(2,q)$. In Section 3, we present the proof of Theorem \ref{conj_main}, which is divided into a series of technical lemmas.

\section{Preliminaries}

Let $q$ be a prime power, and $\F_{q^2}$ be a finite field with $q^2$ elements. We regard $\F_{q^2}$ as a vector space $V$ of dimension $2$ over $\F_q$.  Let  $f:\,\F_{q^2}\rightarrow\F_{q^2}$ be an $\F_q$-linear transformation. There exists a unique polynomial $F(X)=aX+bX^q\in\F_{q^2}[X]$ such that $f(x)=F(x)$ for all $x\in\F_{q^2}$, cf. \cite[Chapter 3]{Finitefield}. Such a polynomial $F(X)$ is a \textit{reduced $q$-linearized polynomial}, i.e., $\deg(F)\le q^2-1$ and the transformation $x\mapsto F(x)$ over $\F_{q^2}$ is $\F_q$-linear.
\begin{lemma}\label{lem_inv_ab}
For $a,b\in\F_{q^2}$, the $\F_q$-linear transformation $x\mapsto ax+bx^q$ over $\F_{q^2}$ is invertible if and only if $a^{q+1}\ne b^{q+1}$.
\end{lemma}
\begin{proof}
The case where one of $a,b$ is zero is trivial, so we assume that $ab\ne 0$.
By \cite[Lemma 7.1]{Finitefield}, the linear transformation $x\mapsto ax+bx^q$ over $\F_{q^2}$ is invertible if and only if $ax+bx^q=0$ has no nonzero root in $\F_{q^2}$. For $x\in\F_{q^2}^*$, $ax+bx^q=0$ if and only if $x^{q-1}=-ab^{-1}$. The latter equation has a solution if and only if $(-ab^{-1})^{q+1}=1$, i.e., $a^{q+1}=b^{q+1}$. The claim now follows.
\end{proof}

By Lemma \ref{lem_inv_ab}, the set of invertible  $\F_q$-linear transformations of $V=\F_{q^2}$, i.e., $\GL(V)$, consists of the transformations of the form
\begin{equation}\label{eqn_fab}
f_{a,b}:\,\F_{q^2}\rightarrow\F_{q^2},\quad x\mapsto ax+bx^q,
\end{equation}
where $a,b\in\F_{q^2}$ such that $a^{q+1}\ne b^{q+1}$.
\begin{lemma}\label{lem_inver}
Suppose that $a,b\in\F_{q^2}$ such that $a^{q+1}\ne b^{q+1}$, and let $f_{a,b}$ be as in \eqref{eqn_fab}. The inverse transformation $f^{-1}$ is $f_{a',b'}$, where
\[
a'=\frac{a^q}{a^{q+1}-b^{q+1}},\quad b'=\frac{-b}{a^{q+1}-b^{q+1}}.
\]
\end{lemma}
\begin{proof}
This follows by direct check, and we omit the details.
\end{proof}

Take an element $t\in \F_{q^2}^*$. For $f_{c,d}\in \GL(V)$, we compute that
\begin{align}
f_{c,d}^{-1}( f_{t,0}(f_{c,d}(x)))
=&\frac{1}{c^{q+1}-d^{q+1}} \left( c^q(tcx+tdx^q)- d(tcx+tdx^q)^q \right) \notag\\
=&\frac{tc^{q+1}-t^qd^{q+1}}{c^{q+1}-d^{q+1}}x+\frac{c^qd(t-t^q)}{c^{q+1}-d^{q+1}}x^q,\label{eqn_ft0_conj}
\end{align}
where we used Lemma \ref{lem_inver} in the first equality. If $d=0$, then it equals $f_{t,0}(x)$. If $d\ne 0$, we write $s:=cd^{-1}$, so that $f_{c,d}^{-1}( f_{t,0}(f_{c,d}(x)))=L_{t,s}(x)$, where
\begin{equation}\label{eqn_Lts}
 	L_{t,s}(x):=\frac{ts^{q+1}-t^q}{s^{q+1}-1}x+\frac{(t-t^q)s^q}{s^{q+1}-1}x^q.
\end{equation}
Here, we have $s^{q+1}\ne 1$ by Lemma \ref{lem_inv_ab}. In particular, we have $L_{t,0}(x)=t^qx$. To be consistent, we write
\begin{equation}\label{eqn_Lts_ing}
 	L_{t,\infty}(x):=tx,\;\;\textup{i.e.,}\;\; L_{t,\infty}=f_{t,0}.
\end{equation}


\begin{lemma}\label{lem_Ct}
For $t\in\F_{q^2}^*$, the conjugacy class of $\GL(V)$ containing $L_{t,\infty}=f_{t,0}$ is
$C_t:=\{L_{t,s}:\,s\in S'\}$, where $S'=\{u\in\F_{q^2}:\,u^{q+1}\ne 1\}\cup\{\infty\}$.
If $t\not\in\F_q$, then $C_t=C_{t^q}$ and $|C_t|=q(q-1)$.
\end{lemma}
\begin{proof}
The first part has been established in the preceding arguments. Assume that $t\not\in\F_q$. We have $L_{t,0}(x)=t^qx$, which is in the conjugacy class $C_t$ and also equals $L_{t^q,\infty}$, so $C_t=C_{t^q}$. For $s\in\F_{q^2}^*$ with $s^{q+1}\ne 1$,  the two transformations $L_{t,s}$ and $L_{t,\infty}$ are distinct since they have different degrees. Also, $L_{t,0}\ne L_{t,\infty}$, since $t\not\in\F_q$.  We deduce from the preceding arguments that the stabilizer of $L_{t,\infty}$ under the action of $\GL(V)$ via conjugation is $\{f_{c,0}:\,c\in\F_{q^2}^*\}$. Therefore,
\[
|C_t|=\frac{|\GL(V)|}{q^2-1}=\frac{(q^2-1)(q^2-q)}{q^2-1}=q(q-1).
\]
\end{proof}

\begin{lemma}\label{lem_char_X}
For $g\in \GL(V)$ with $V=\F_{q^2}$, its quotient image $\bar{g}$ in $\PGL(V)$ is in the set $A$ in Theorem \ref{conj_main} if and only if it is conjugate to $f_{t,0}$ in $\GL(V)$ for some element $t\in\F_{q^2}^*$ such that $t^{2(q-1)}\ne 1$.
\end{lemma}
\begin{proof}
Take $g\in \GL(V)$ such that $\bar{g}\in A$, i.e., $\bar{g}^{q+1}=1$, $\bar{g}^2\ne 1$. Then $g^{q+1}$ is the scalar multiplication by some $\lambda\in\F_q^*$. In particular, $g^{q^2-1}=1$.  In this proof, we regard $g$ as a $2\times 2$ matrix over $\F_q$ by specifying a basis of $V=\F_{q^2}$ over $\F_q$. Let $m(X)$ be the minimal polynomial of the matrix $g$, which lies in $\F_q[X]$. Since $g^{q^2-1}=1$, we have $m(X)|X^{q^2-1}-1$. Since $X^{q^2-1}-1$ has no repeated roots, so does $m(X)$. Also, by the Cayley-Hamilton Theorem, we have $m(g)=0$.

We claim that $m(X)$ is an irreducible polynomial of degree $2$ over $\F_q$.  If $m(X)$ has degree $1$, then $g$ is a scalar matrix, and $\bar{g}=1$, which is impossible.  If $m(X)$ has two roots $\lambda_1,\lambda_2$ in $\F_q$, then $\lambda_1\ne\lambda_2$ by the previous paragraph and so $g$ is conjugate to the diagonal matrix $\diag(\lambda_1,\lambda_2)$ in $\GL(V)$. Since $g^{q+1}$ is a scalar matrix, so should be $\diag(\lambda_1^{q+1},\lambda_2^{q+1})$. We thus have $\lambda_1^2=\lambda_2^2$, i.e., $\lambda_1=\pm\lambda_2$. This leads to a contradiction that $\bar{g}^2=1$, and thus the claim follows.

Write $m(X)=X^2+cX+d$ with $c,d\in\F_q$, and let $t$ be a root of $m(X)$ in $\F_{q^2}$. Then $t\not\in\F_q$, and $m(X)=(X-t)(X-t^q)$.  Set $v:=g(1)$. We claim that $v\not\in \F_{q}$: otherwise, $g(1)=v\cdot 1$ and $v$ is an eigenvalue of $g$, so $m(v)=0$, contradicting the irreducibility of $m(X)$. We have $g(g(1))+cg(1)+d=0$, so $g(v)=-d-cv$. Therefore, $(g(1),g(v))=(1,v)M$ with $M=\begin{pmatrix}0&-d\\1&-c\end{pmatrix}$. It is routine to check
$(f_{t,0}(1),f_{t,0}(t))=(1,t)M$. Let $h$ be the element of $\GL(V)$ that maps the ordered basis $(1,t)$ to $(1,v)$. Then $g$ is mapped to $f_{t,0}$ by $h$ via conjugation, i.e., $g$ is in the conjugacy class of $f_{t,0}$ in $\GL(V)$.

We claim that $t^{2(q-1)}\ne 1$. Otherwise, $t^2\in\F_q^*$ and $X^2-t^2$ is the minimal polynomial of $t$ over $\F_q$. It follows that $m(X)=X^2-t^2$. Hence $g(g(x))=t^2x$, and $\bar{g}^2=1$: a contradiction. This proves the necessity part of the lemma.

For the sufficient part, it suffices to verify that $\overline{f_{t,0}}^{q+1}=1$ and $\overline{f_{t,0}}^{2}\ne 1$ provided that $t^{2(q-1)}\ne 1$. This is clear since $f_{t,0}^{q+1}(x)=t^{q+1}x$, $f_{t,0}^{2}(x)=t^{2}x$, and $t^{q+1}$ is in $\F_q^*$ while $t^2$ is not under the assumption on $t$. This completes the proof.
\end{proof}

The following is an immediate corollary.
\begin{corollary}\label{cor_char_X}
Let $A$ be the set as in Theorem \ref{conj_main}, regard $\F_{q^2}$ as a $2$-dimensional vector space over $\F_q$ so that $\GL(2,q)=\GL(V)$. Then the full preimage of $A$ in $\GL(V)$
is the set $\tilde{A}:=\cup_{t\in T}C_t$, where $T=\{t\in\F_{q^2}^*:\,t^{2(q-1)}\ne 1\}$.
\end{corollary}

\begin{lemma}\label{lem_Ct_conj}
Take the same notation as in Lemma \ref{lem_Ct}. For $t,t'\in \F_{q^2}^*$, we have $C_t=C_{t'}$ if and only if $t'=t$ or $t'=t^q$.
\end{lemma}
\begin{proof}
We have $C_t=C_{t^q}$ by Lemma \ref{lem_Ct}. The minimal polynomial of the $\F_q$-linear transformation  $f_{t,0}$ is $X-t$ or $(X-t)(X-t^q)$ according as $t$ is in $\F_q$ or not, as we showed in the proof of Lemma \ref{lem_char_X}. Since conjugate transformations has the same minimal polynomial, the claim now follows.
\end{proof}

We shall need the following technical lemma in the next section.
\begin{lemma}\label{lem_disc}
Suppose that $q$ is odd, and let $x_0$ be a root of $F(x)=cx^2+dx+c^q$, where $c\in\F_{q^2}^*$, $d\in\F_q^*$. Then $x_0^{q+1}=1$ if and only if $\Delta_F=d^2-4c^{q+1}$ is zero or a nonsquare in $\F_q$.
\end{lemma}
\begin{proof}
Take $r\in\F_{q^2}^*$ such that $r^2=\Delta_F$. This is possible since  $\Delta_F$ is in $\F_q$. We solve that $x_0=\frac{-d+\epsilon r}{2c}$ for some $\epsilon=\pm 1$, so
\begin{align*}
(2c)^{q+1}x_0^{q+1}=(-d+\epsilon r)(-d+\epsilon r^q)
=d^2-d\epsilon(r+r^q)+r^{q+1}.
\end{align*}
We consider three separate cases.
\begin{enumerate}
\item[(1)] If $\Delta_F=0$, then $4c^{q+1}=d^2$, $r=0$, and $4c^{q+1}x_0^{q+1}=d^2$, so $x_0^{q+1}=1$.
\item[(2)] If $\Delta_F$ is a nonsquare in $\F_q^*$, then the minimal polynomial of $r$ over $\F_q$ is $X^2-\Delta_F$, and so $r^q+r=0$, $r^{q+1}=-\Delta_F$. In this case, $4c^{q+1}x_0^{q+1}=d^2-\Delta_F=4c^{q+1}$, and we have $x_0^{q+1}=1$.
\item[(3)] If $\Delta_F$ is a square in $\F_q^*$, then $r\in\F_q^*$. In this case, $4c^{q+1}x_0^{q+1}=d^2-2d\epsilon r+r^2$. If $x_0^{q+1}=1$, then we deduce that $-2d\epsilon r+2r^2=0$, i.e., $d\epsilon=r$. It follows that $d^2=r^2=d^2-4c^{q+1}$, yielding $c=0$: a contradiction. Hence $x_0^{q+1}\ne 1$ in this case.
\end{enumerate}
This completes the proof.
\end{proof}

\section{Proof of Theorem \ref{conj_main}}
This section is devoted to the proof of Theorem \ref{conj_main}, which is divided into a series of lemmas. We regard $\F_{q^2}$ as a two-dimensional vector space $V$ over $\F_q$, and identify $\GL(2,q)$ with $\GL(V)$. The cases $q=2,3$ of Theorem \ref{conj_main} can be directly verified by computer, so we assume that $q>3$ in the sequel. We introduce the following subset of $\GL(V)$:
\begin{align*}
\tilde{D}:=\{f_{t,0},\,f_{0,t}:\,t\in\F_{q^2}^*\},
\end{align*}
where $f_{a,b}$ is as in \eqref{eqn_fab}. It is a subgroup of order $2(q^2-1)$, and its quotient image in $\PGL(V)$ is a dihedral subgroup $D$ of order $2(q+1)$. When $q>3$, there is exactly one conjugacy class of dihedral subgroups of order $2(q+1)$ in $\PGL(2,q)$, cf. \cite{Dickson1901,King2005}. We can thus take this $D$ as the one in
Theorem \ref{conj_main}.

It is convenient to introduce the following notation:
\begin{align}
S&:=\{s\in\F_{q^2}^*:\,s^{q+1}\ne 1\},\label{eqn_S}\\
T&:=\{t\in\F_{q^2}^*:\,t^{2(q-1)}\ne 1\}.\label{eqn_T}
\end{align}
We have $\F_q^*\cap T=\emptyset$, and $|T|=(q^2-1)-(q-1)=q(q-1)$ or $|T|= (q^2-1)-2(q-1)=q(q-1)$ according as $q$ is even or odd. The full preimage $\tilde{A}$ of the set $A$ in Theorem \ref{conj_main} is the union of conjugacy class $C_t$'s by Corollary \ref{cor_char_X}, where $t$ ranges over $T$. Let $T_1$ be a subset of $T$ such that it contains exactly one element of $\{t,t^q\}$ for each $t\in T$. Then $|T_1|=\frac{1}{2}|T|$, and $\tilde{A}$ is the disjoint union of $C_t$'s with $t\in T_1$ by Lemma \ref{lem_Ct_conj}.

We now reformulate Theorem \ref{conj_main} in terms of $\GL(V)$. To show that $A\cdot D=\lambda \cdot \PGL(2,q)$, where $\lambda=q$ or $q-1$ according as $q$ is even or odd, it suffices to show that $\tilde{A}\cdot\tilde{D}=|T|\cdot \GL(V)$. Recall that the conjugacy class $C_t$ consists of $L_{t,s}$'s with $s\in S\cup\{0\}\cup\{\infty\}$, cf. \eqref{eqn_Lts} and \eqref{eqn_Lts_ing}.

\begin{lemma}\label{lem_str_ab0}
For $s\in S$, $t\in T$ and $r\in\F_{q^2}^*$, if $f_{a,b}=L_{t,s}\cdot f_{r,0}$ or $f_{a,b}=L_{t,s}\cdot f_{0,r}$, then $ab\ne 0$.
\end{lemma}
\begin{proof}
We have $L_{t,s}(f_{r,0}(x))=\frac{(ts^{q+1}-t^q)r}{s^{q+1}-1}x+\frac{(t-t^q)s^qr^q}{s^{q+1}-1}x^q$. If the coefficient of $x$ is zero, then $ts^{q+1}=t^q$, i.e., $t^{q-1}=s^{q+1}$. We deduce that $t^{(q-1)^2}=1$. Since $t^{q^2-1}=1$, we deduce that $t^{2(q-1)}=t^{q^2-1-(q-1)^2}=1$, a contradiction to $t\in T$. If the coefficient of $x^q$ is zero, then $t=t^q$, i.e., $t^{q-1}=1$, which is again a contradiction to $t^{2(q-1)}\ne 1$. The case of  $f_{a,b}=L_{t,s}\cdot f_{0,r}$ is similar, and we omit the details.
\end{proof}

\begin{lemma}\label{lem_mult_fab0}
For $f_{a,b}\in\GL(V)$ with $ab=0$, its multiplicity in $\tilde{A}\cdot\tilde{D}$ is $|T|$, which is $q(q-1)$ or $(q-1)^2$ according as $q$ is even or odd.
\end{lemma}
\begin{proof}
We consider only the case $f_{a,0}$ with $a\in\F_{q^2}^*$, and the case of $f_{0,b}$ with $b\in\F_{q^2}^*$ is similar. By Lemma \ref{lem_str_ab0}, it suffices to consider the number of ways to express $f_{a,0}$ in the form $L_{t,0}\cdot g$ with $g\in\tilde{D}$, $t\in T_1$.
\begin{enumerate}
\item[(1)]If $s=\infty$, $L_{t,\infty}(f_{r,0}(x))=trx$, $L_{t,\infty}(f_{0,r}(x))=trx^q$.
\item[(2)]If $s=0$, then $L_{t,0}(f_{r,0}(x))=t^qrx$, $L_{t,0}(f_{0,r}(x))=t^qrx^q$.
\end{enumerate}
We recall that $t\ne t^q$ for $t\in T$. For each $t\in T_1$, we  have $f_{a,0}=L_{t,\infty}\cdot f_{at^{-1},0}=L_{t,0}\cdot f_{at^{-q},0}$. This gives $2|T_1|$ different expressions of $f_{a,0}$. This completes the proof.
\end{proof}

It remains to consider the multiplicity of $f_{a,b}$ in $\tilde{A}\cdot\tilde{D}$, where $f_{a,b}\in GL(V)$ with $ab\ne 0$. By Lemma \ref{lem_str_ab0} and the proof of Lemma \ref{lem_mult_fab0}, it equals the number of $(s,t,r)\in S\times T_1\times\F_{q^2}^*$ such that $L_{t,s}\cdot f_{r,0}=f_{a,b}$ or $L_{t,s}\cdot f_{0,r}=f_{a,b}$. By comparing the coefficients of $x$ and $x^q$, the two equations are equivalent to
\begin{equation}\label{eqn_coeff_ab}
a=\frac{(ts^{q+1}-t^q)r}{s^{q+1}-1},\,
b=\frac{(t-t^q)s^qr^q}{s^{q+1}-1},
\end{equation}
and
\begin{equation}\label{eqn_coeff_ab2}
b=\frac{(ts^{q+1}-t^q)r}{s^{q+1}-1},\,
a=\frac{(t-t^q)s^qr^q}{s^{q+1}-1},
\end{equation}
respectively. Observe that \eqref{eqn_coeff_ab2} can be obtained from \eqref{eqn_coeff_ab} by switching $a,\,b$. We analyze \eqref{eqn_coeff_ab} in the sequel, and the conclusions hold for \eqref{eqn_coeff_ab2} by interchanging $a,b$. We deduce from \eqref{eqn_coeff_ab} that
\begin{equation}\label{eqn_ttq}
t=a r^{-1}-br^{-q}s^{-q},\quad
t^q=a r^{-1}-b r^{-q}s.
\end{equation}
Raising the expression of $t$ to $q$-th power and comparing with that of $t^q$, we deduce that $s$ is a root of the quadratic equation
\begin{equation}\label{eqn_quad}
brX^2+(a^qr-ar^q)X-(br)^{q}=0.
\end{equation}
\begin{lemma}\label{lem_quad_sol}
The quadratic equation \eqref{eqn_quad} has all its roots in $\F_{q^2}^*$.
\end{lemma}
\begin{proof}
If $q$ is odd, its discriminant $\Delta=(a^qr-ar^q)^2+4(br)^{q+1}$ lies in $\F_q$ upon direct check, so the square roots of $\Delta$ and correspondingly the roots of \eqref{eqn_quad} lie in $\F_{q^2}$. If $q$ is even, we can rewrite the equation as $Y^2+(a^qr+ar^q)Y+(br)^{q+1}=0$, where $Y=brX$. The latter equation has coefficients in $\F_q$, so its roots lie in $\F_{q^2}$. This completes the proof.
\end{proof}

\begin{lemma}\label{lem_quad_viete}
If $s\in S$ is a root of \eqref{eqn_quad}, then
\begin{enumerate}
\item[(1)] $s^{-q}$ is also a root and $s+s^{-q}=-b^{-1}a^q+b^{-1}ar^{q-1}$, $s^{1-q}=(br)^{q-1}$.
\item[(2)] If the values of $t$ corresponding  to $s$ and $s^{-q}$ as in \eqref{eqn_ttq} are in $T$, then exactly one of them is in $T_1$.
\end{enumerate}
\end{lemma}
\begin{proof}
(1) Suppose $brs^2+(a^qr-ar^q)s-(br)^{q}=0$. Raising both sides to the $q$-th power and then multiplying both sides by $-s^{-2q}$, we get $-(br)^q+(a^qr-ar^q)s^{-q}+brs^{-2q}=0$. The  first claim now follows. Since $s\in S$, we have $s\ne s^{-q}$ and they are the two distinct roots of \eqref{eqn_quad}. The second claim then follows from Vi$\grave{e}$te Theorem.

(2) Let $t_1=a r^{-1}-br^{-q}s^{-q}$ and $t_2=a r^{-1}-br^{-q}s$ be the two corresponding values of $t$. Then $t_1^q=t_2$ by \eqref{eqn_ttq}. The claim now follows.
\end{proof}

\begin{lemma}\label{lem_ts}
Take $(s,t,r)\in S\times \F_{q^2}^*\times\F_{q^2}^*$ such that \eqref{eqn_coeff_ab} holds.
If $q$ is even, then $t\in T$; if $q$ is odd, then $t\in T$ if and only if $r^{q-1}+a^{q-1}\ne0$.
\end{lemma}
\begin{proof}
Recall that  $t\not\in T$ if and only if $t^2\ne t^{2q}$, cf. \eqref{eqn_T}. From \eqref{eqn_ttq} we deduce that
\begin{align*}
t^{2q}-t^2&=(ar^{-1}-br^{-q}s)^2-(ar^{-1}-br^{-q}s^{-q})^2 \\
             &=-2abr^{-q-1}s+b^2r^{-2q}s^2+2abr^{-q-1}s^{-q}-b^2r^{-2q}s^{-2q}\\
            &=br^{-q}(s-s^{-q})(-2ar^{-1}+br^{-q}(s+s^{-q})).
\end{align*}
Since $s^{q+1}\ne 1$, i.e., $s\ne s^{-q}$, we see that $t^{2q}-t^2=0$ iff   $-2ar^{-1}+br^{-q}(s+s^{-q})=0$.  This does not hold if $q$ is even, so $t$ is in $T$ in this case. If $q$ is odd, then  $-2ar^{-1}+br^{-q}(s+s^{-q})=-ar^{-q}(r^{q-1}+a^{q-1})$ by Lemma \ref{lem_quad_viete}, and the claim follows in this case.
\end{proof}

By Lemma \ref{lem_quad_sol} and Lemma \ref{lem_quad_viete}, for an element $r\in\F_{q^2}^*$ the two roots of \eqref{eqn_quad} are either both in $S$ or both in $\F_{q^2}^*\setminus S$, and in the former case the two roots are distinct. Set
\begin{align}
R_1:&=\{r\in\F_{q^2}^*:\,\eqref{eqn_quad} \textup{ has no roots in } S\},\\
R_2:&=\{r\in\F_{q^2}^*:\,\eqref{eqn_quad} \textup{ has two roots in } S\}.
\end{align}
The two subsets $R_1$ and $R_2$ form a partition of $\F_{q^2}^*$.
\begin{lemma}\label{lnm_cor_1}
Suppose that $q$ is odd and $r^{q-1}+a^{q-1}=0$. Then $r\in R_2$ if and only if $\kappa(b,a):=-1+b^{q+1}/a^{q+1}$ is a nonzero square in $\F_q$.
\end{lemma}
\begin{proof}
We multiply both sides of \eqref{eqn_quad} by $r^{-1}a$ and
substitute $r^{q-1}=-a^{q-1}$ to obtain $baX^2+2a^{q+1}X+(ba)^q=0$.  Its discriminant is $\Delta=4a^{2q+2}\kappa(b,a)$.  By Lemma \ref{lem_disc} and the  preceding remark, its roots lie in $S$ if and only if $\Delta$ is a nonzero square in $\F_q$. This completes the proof.
\end{proof}

\begin{lemma}\label{lem_R1R2}
Take notation as above, and set $\kappa(a,b):=1-a^{q+1}/b^{q+1}$.
\begin{enumerate}
\item[(1)] If $q$ is odd and $\kappa(a,b)$ is a nonsquare in $\F_q^*$, then $|R_1|=\frac{(q-1)(q+3)}{2}$, $|R_2|=\frac{(q-1)^2}{2}$;
\item[(2)] If $q$ is odd and $\kappa(a,b)$ is a square in $\F_q^*$, then $|R_1|=|R_2|=\frac{q^2-1}{2}$;
\item[(3)] If $q$ is even, then $|R_1|=\frac{(q+2)(q-1)}{2}$, $|R_2|=\frac{q(q-1)}{2}$.
\end{enumerate}
\end{lemma}
\begin{proof}

Suppose that $s$ is an element of $\F_{q^2}^*$ such that $s^{q+1}=1$ and it is a root of \eqref{eqn_quad} for some $r\in\F_{q^2}^*$. Then we have $bs^2+a^qs=(as+b^q)r^{q-1}$. Since $a^{q+1}\ne b^{q+1}$ and $s^{q+1}=1$, we deduce that $as+b^q\ne 0$. Similarly, $a+b^qs^q\ne 0$. Hence, $r^{q-1}$ equals
\begin{align}\label{eqn_rqm1}
\frac{bs^2+a^qs}{as+b^q}=(a+b^qs^q)^{q-1}.
\end{align}
Therefore, $|R_1|=(q-1)|Y|$, where $Y:=\{\frac{bs^2+a^qs}{as+b^q}:\,s^{q+1}=1\}$, and $|R_2|=q^2-1-|R_1|$. We note that $Y$ consists of $(q-1)$-st powers by  \eqref{eqn_rqm1}, i.e., $c^{q+1}=1$ for $c\in Y$.

We write $Y'$ for the multiset of size $q+1$ corresponding to $Y$. Each element $c\in Y'$ has multiplicity $1$ or $2$, since
\begin{equation}\label{eqn_xc}
bX^2+a^qX=c(aX+b^q)
\end{equation}
has at most two solutions. Let $n_i$ be the number of elements with multiplicity $i$ in $Y'$ for $i=1,2$. Then $n_1+2n_2=q+1$, $|Y|=n_1+n_2$. It remains to determine $n_1$.

We claim that $c$ has multiplicity $1$ in $Y'$ if and only if \eqref{eqn_xc} has a repeated root. Suppose that \eqref{eqn_xc} has two distinct roots $s,s'$ for some $c\in Y$ and $s$ such that $s^{q+1}=1$. Then $ss'=-cb^{q-1}$ by Vi$\grave{e}$te Theorem. By raising to the $(q+1)$-st power, we deduce that $s'^{q+1}=1$, and so $c$ has multiplicity $2$ in $Y'$. Conversely, suppose that \eqref{eqn_xc} has a repeated root $s$. Since each element of $Y'$ is a $(q-1)-$st power, we write $c=z^{q-1}$ for some $z\in\F_{q^2}^*$. Take $\delta\in\F_{q^2}$ such that $\delta^q+\delta=0$. Multiplying both sides of \eqref{eqn_xc} by $\delta z$, we obtain $(bz\delta)X^2+\tr_{\F_{q^2}/\F_q}(a^qz\delta)X+(bz\delta)^q=0$. If $q$ is odd, then $s^{q+1}=1$ by Lemma \ref{lem_disc}. If $q$ is even, the coefficient  $a^q-ac$ of $X$ in \eqref{eqn_xc} equals $0$ since the equation has a repeated root. Thus \eqref{eqn_xc} has the form $X^2=cb^{q-1}=(zb)^{q-1}$, so the repeated root $s$ satisfies $s^{q+1}=1$. This proves the claim.

To summarize, the elements that have multiplicity $1$ in $Y'$ are those $c\in\F_{q^2}^*$ such that $c$ is a $(q-1)$-st power and \eqref{eqn_xc} has a repeated root.

Suppose that $q$ is odd. The equation \eqref{eqn_xc} has a repeated root if and only if its determinant $\Delta_c:=(a^q-ca)^2+4cb^{1+q}$ is $0$, i.e., $c$ is a root of
\begin{equation}
a^2X^2+(-2a^{q+1}+4b^{1+q})X+a^{2q}=0.
\end{equation}
Similar to the proof of Lemma \ref{lem_quad_sol}, it has two solutions in $\F_{q^2}^*$. By Lemma \eqref{lem_disc}, its solutions are $(q-1)$-st power if and only if its discriminant $16b^{2(q+1)}\kappa(a,b)$ is $0$ or a nonsquare of $\F_q^*$. Since $a^{q+1}\ne b^{q+1}$, we have $\kappa(a,b)\ne 0$. Therefore, $n_1=0$ or $2$ according as $\kappa(a,b)$ is a square or nonsquare of $\F_{q}^*$. The claims in (1) and (2) now follows.

Suppose that $q$ is even. The equation \eqref{eqn_xc} has a repeated root if and only if its coefficient of $X$ is $0$, i.e., $a^q=ac$. It follows that $c=a^{q-1}$, and so $n_1=1$. The claim in (3) then follows. This completes the proof.
\end{proof}

Let $N_{a,b}$ be the number of triples $(s,t,r)\in S\times T_1\times\F_{q^2}^*$ such that
\eqref{eqn_coeff_ab} holds. Then $N_{b,a}$ is the number of triples such that
\eqref{eqn_coeff_ab2} holds. We are now ready to compute $N_{a,b}+N_{b,a}$. Let $\square$ (resp. $\blacksquare$) be the set of nonzero squares and nonsquares of $\F_q$. For a property $P$, we define $[[P]]:=1$ or $0$ according as $P$ holds or not.

\begin{lemma}\label{lem_Nab}
Take notation as above. We have
\begin{equation*}
N_{a,b}=\begin{cases}(q-1)\cdot\left(\frac{q-1}{2}+[[\kappa(a,b)\in\square]]-[[\kappa(b,a)\in\square]]\right),\;&\textup{ if $q$ is odd};\\\frac{q}{2}(q-1), &\textup{ if $q$ is even}.\end{cases}
\end{equation*}
\end{lemma}
\begin{proof}
We first consider the case $q$ is odd. Take a triple $(s,t,r)\in S\times T_1\times\F_{q^2}^*$ such that \eqref{eqn_coeff_ab} holds. By the analysis preceding Lemma \ref{lem_quad_sol}, the value of $t$ is uniquely determined by $r$ and $s$ by \eqref{eqn_ttq}, and \eqref{eqn_quad} should have two solutions in $S$, i.e., $r\in R_2$. By Lemma \ref{lem_ts}, we have $r^{q-1}+a^{q-1}\ne 0$. By Lemma \ref{lnm_cor_1}, an element $r$ such that $r^{q-1}+a^{q-1}\ne 0$ is in $R_2$ if and only if $\kappa(b,a)\in\square$.

We now reverse the above arguments. Take $r\in R_2$ such that $r^{q-1}+a^{q-1}\ne 0$. The two solutions $s_1,s_2$ of \eqref{eqn_quad} are both in $S$, and $t_i=a r^{-1}-br^{-q}s_i^{-q}$, $i=1,2$, are in $T$ by Lemma \ref{lem_ts}. By (2) of Lemma \ref{lem_quad_viete}, exactly one of $t_i$'s is in $T_1$, say, $t_1$. Then $(s_1,t_1,r)$ is a solution to \eqref{eqn_coeff_ab}. Therefore, $N_{a,b}=|R_2|-(q-1)[[\kappa(b,a)\in\square]]$, where the latter term corresponds to the $q-1$ $r$'s such that $r^{q-1}+a^{q-1}=0$. By Lemma \ref{lem_R1R2}, we have $|R_2|=(q-1)\cdot\left(\frac{q-1}{2}+[[\kappa(a,b)\in\square]]\right)$. To sum up, we have
\[
N_{a,b}=(q-1)\cdot\left(\frac{q-1}{2}+[[\kappa(a,b)\in\square]]\right)-(q-1)[[\kappa(b,a)\in\square]].
\]
This completes the proof of the case $q$ is odd. By the same argument, we get $N_{a,b}=|R_2|$ in the case $q$ is even. This completes the proof.
\end{proof}

\begin{corollary}\label{cor_mult_fabn0}
For $f_{a,b}\in\GL(V)$ with $ab\ne 0$, its multiplicity in $\tilde{A}\cdot\tilde{D}$ equals $q(q-1)$ or $(q-1)^2$ according as $q$ is even or not.
\end{corollary}
\begin{proof}
By our analysis following Lemma \ref{lem_mult_fab0}, this number equals the number of triples $(s,t,r)\in S\times T_1\times\F_{q^2}^*$ such that \eqref{eqn_coeff_ab} or \eqref{eqn_coeff_ab2} holds. It equals $N_{a,b}+N_{b,a}$, so the claim follows from Lemma \ref{lem_Nab}.
\end{proof}

By combining Lemma \ref{lem_mult_fab0} and Corollary \ref{cor_mult_fabn0}, we see that $\tilde{A}\cdot\tilde{D}=\lambda(q-1)\cdot\GL(V)$, where $\lambda=q$ or $q-1$ according as $q$ is even or odd. By taking quotient by the center, we conclude that Theorem \ref{conj_main} holds.

\vspace*{10pt}

\noindent\textbf{Acknowledgement.} This work was supported by National Natural Science Foundation of China under Grant No. 11771392.\\
\begin{center}
	\scriptsize
	\setlength{\bibsep}{0.5ex}  
		\linespread{0.5}
	\bibliographystyle{plain}
	
	\footnotesize

\end{center}

\end{document}